\documentclass{amsart}
\usepackage{amsmath,amssymb,amsthm,amsfonts}
\usepackage{graphicx} 
\usepackage[]{currvita}
\usepackage{subfiles}
\usepackage{enumitem}
\usepackage[english]{babel}
\usepackage{tabularx}
\usepackage[utf8]{inputenc}
\usepackage[T1]{fontenc}
\usepackage{mathtools}
\usepackage{pstricks-add}
\usepackage{lmodern,dsfont}
\usepackage{enumerate}
\newtheorem{definition}{Definition}[section]
\newtheorem{remark}{Remark}[section]
\newtheorem{defi}[definition]{Definition}
\newtheorem{prop}[definition]{Proposition}

\newtheorem{lem}[definition]{Lemma}
\newtheorem{thm}[definition]{Theorem}
\newtheorem{cor}[definition]{Corollary}
\newtheorem{corand}[definition]{Corollary and Definition}
\theoremstyle{definition}
\newtheorem{rem}[definition]{Remark}
\newtheorem{ex}[remark]{Example}
\usepackage{fixltx2e,mparhack,url,varioref,}
\usepackage{esvect}
\usepackage{enumitem} 
\usepackage{mdframed} 
\newmdenv[
  topline=false,
  bottomline=false,
  skipabove=\topsep,
  skipbelow=\topsep
]{siderules}
\mdfsetup{
    linewidth=0.6pt
}
\usepackage{tikz-cd}
\usepackage{bm}
\usepackage{textcomp}
\makeatletter
\def\blfootnote{\xdef\@thefnmark{}\@footnotetext}
\makeatother
\title{A duality for Guichard nets}
\author{Gudrun Szewieczek}
\begin{document}
\maketitle
\begin{center}
\begin{minipage}{11cm}\small
\textbf{Abstract.} In this paper we study G-surfaces, a rather unknown surface class originally defined by Calapso, and show that the coordinate surfaces of a Guichard net are G-surfaces. Based on this observation, we present distinguished Combescure transformations that provide a duality for Guichard nets. Another class of special Combescure transformations is then used to construct a B\"acklund-type transformation for Guichard nets. In this realm a permutability theorem for the dual systems is proven. 
\end{minipage}
\vspace*{0.5cm}\\\begin{minipage}{11cm}\small
\end{minipage}
\vspace*{0.5cm}\\\begin{minipage}{11cm}\small
\end{minipage}
\end{center}
\  \\ 
%
%
\blfootnote{G.\,SZEWIECZEK,  gudrun@geometrie.tuwien.ac.at}
\blfootnote{TU Wien, Wiedner Hauptstr.\,8-10/104, 1040 Wien, Austria.}
\section{Introduction}
\noindent Various integrable surface classes were classically characterized by the existence of particular Combescure transformed dual surfaces: for example, isothermic surfaces via the Christoffel transforms and Guichard surfaces by the Guichard duality (\cite{calapso_guichard, christoffel, eisenhart_trafo_book, guichard_defi}). Based on this principle, in \cite{O_surface} the theory of O-surfaces was established, which provides a unified way to describe many classically known integrable surface classes.

Recently \cite{pember_lie_applicable}, a generalization of Demoulin's duality for $\Omega$-surfaces \cite{demoulin_associated} has been found, revealing that the Lie applicable class of $\Omega$-surfaces also appears as O-surfaces. Moreover, this concept also applies to higher dimensional hypersurfaces as the example of 3-dimensional conformally flat hypersurfaces shows (cf.\,\cite{duality, my_thesis}). 
\\\\Although triply orthogonal systems possess a rich transformation theory (\cite{bianchi_1918, bianchi_1919, Darboux_ortho, Ganzha_1996}), the concept of characterizing Combescure transformations to distinguish special subclasses of systems has not been exploited yet. 
In this paper we adopt this approach and present a duality for the subclass of Guichard nets that allows, together with particular associated systems, a characterization of Guichard nets in terms of their Combescure transforms. 
%
\\\\Guichard nets, that is, triply orthogonal systems with an induced metric fulfilling a special trace-zero-condition, were commonly studied by classical geometers around 1900 (\cite{Darboux_ortho, Guichard_small, salkowski}). Renewed interest arises from the observation that Guichard nets provide characterizing curvature line coordinates for 3-dimensional conformally flat hypersurfaces \cite{uhj_guichard}. This has created recent research interest at the junction of these triply orthogonal systems and hypersurfaces (\cite{zbMATH06902436, MR3747518, doReiFilho2018, MR3056177, MR3862798}). In particular, special subclasses such as cyclic Guichard nets \cite{cyclic_guichard, cyclic_guichard_eth}, Bianchi-type Guichard nets \cite{bianchi_type_guichard} and isothermic Guichard nets \cite{iso_guichard, my_thesis} have been investigated.
\\\\The duality for Guichard nets presented in this paper relies on the crucial fact that the coordinate surfaces of a Guichard net are G-surfaces. This surface class was introduced by Calapso in \cite{calapso_G} and then apparently fell into oblivion:  a surface $f$ is a G-surface if there exists a Combescure related associated surface $\hat{f}$ such that the following relation 
\begin{equation*}
cH_1^2H_2^2\Big(\frac{\kappa_1}{\hat{\kappa}_1}-\frac{\kappa_2}{\hat{\kappa}_2}\Big)^2=H_2^2 + \varepsilon H_1^2 \ \ \ \text{for some } \varepsilon \in \{ 0 , \pm 1 \} \ \text{and } c \in \mathbb{R}\setminus \{ 0 \}
\end{equation*}
between the principal curvatures and the coefficients $H_1^2$ and $H_2^2$ of the induced metric of $f$ holds. Note, however, that this relation is not symmetric in $f$ and $\hat{f}$, hence the Combescure transform $\hat{f}$ is in general not a G-surface with associated surface $f$. 

In Section \ref{section_G_surfaces} we demonstrate that for a G-surface there additionally exist Combescure transforms that are again G-surfaces and allow for a dual construction. This characterization immediately reveals that G-surfaces can be interpreted in the framework of O-surfaces and contain isothermic, as well as Guichard surfaces as subclasses.
\\\\These observations will be used in Section \ref{section_guichard_nets} to prove how G-surfaces arise in the context of triply orthogonal systems. In particular, we show that the coordinate surfaces of a Guichard net are G-surfaces.  As a consequence, there exist  associated and dual surfaces for the coordinate surfaces of a Guichard net that, suitably chosen, give rise to new triply orthogonal systems. It is proven that the triply orthogonal systems formed by the dual surfaces are again Guichard nets, hence this dual construction provides a method to construct new Guichard nets from a given one.   
\\\\Another way of obtaining new Guichard nets was given in \cite{bianchi_1918, bianchi_1919}: it has been shown that for any Guichard net there exist Ribaucour transformations that preserve the Guichard condition. We reconsider this B\"acklund-type transformation in Section~\ref{section_trafo} by using the fact that any Ribaucour transformation can be decomposed into two Combescure transformations and an inversion in the unit sphere. In this way, the Ribaucour transformation between two Guichard nets is traced back to the determination of particular Combescure transforms of a Guichard net. Moreover, we prove permutability between this Ribaucour transformation and the duality for Guichard nets.
\\\\\textit{Acknowledgements.} This work originated in the author`s doctoral thesis \cite{my_thesis} and was partially supported by the FWF/JSPS Joint Project grant  \textit{I1671-
N26 Transformations and Singularities}.
Special gratitude is expressed to Udo Hertrich-Jeromin for helpful comments  
and various geometric insights into Guichard nets. Moreover, the author would like to thank Fran Burstall, Mason Pember and Wolfgang Schief for fruitful and enjoyable discussions on G- and O-surfaces.  
%
%
\section{Preliminaries}
\noindent In this work we only deal with surfaces $f=(f_1, f_2, f_3):M^2 \rightarrow \mathbb{R}^3$ parametrized by curvature line coordinates, that is, the coordinates are orthogonal and conjugate. Equivalently, the coordinate functions $f_1, f_2$ and $f_3$, as well as the function $|f|^2$, are solutions of the \emph{point equation of the surface} \cite[Chap.\,I.2, IV.\,61]{eisenhart_trafo_book}
\begin{equation}\label{point_equ_surf}
\partial_{xy} \theta = \partial_y \log H_1 \partial_x \theta + \partial_x \log H_2 \partial_y \theta, 
\end{equation}
where the induced metric of the surface is denoted by $I=H_1^2 dx^2+ H_2^2 dy^2$.

\subsection{Combescure transformation of a surface} Classically \cite{comb, eisenhart_trafo_book}, two surfaces $f$ and $\hat{f}$ are related by a \emph{Combescure transformation} if they have parallel tangent planes along corresponding conjugate coordinate lines. However, note that a Combescure transformation generically preserves curvature line coordinates. Hence, in this paper we consider without loss of generality parallel tangent directions along curvature line coordinates. 

Analytically, any Combescure transform $\hat{f}$ of $f:M^2 \rightarrow \mathbb{R}^3$ is (up to translation) uniquely determined by 
\begin{equation}\label{def_comb}
d \hat{f} = h \partial_x f dx +  l \partial_y f dy,
\end{equation}
where $h, l:M^2\rightarrow \mathbb{R}$ are two functions fulfilling the compatibility conditions
\begin{equation}
\begin{aligned}\label{comp_comb}
\partial_y h &= (l-h) \partial_y \ln H_1,
\\\partial_x l &= (h-l) \partial_x \ln H_2.
\end{aligned}
\end{equation}
Equivalently \cite[Chap.\,I.4]{eisenhart_trafo_book}, any function $\varphi: M^2 \rightarrow \mathbb{R}$ fulfilling
\begin{equation}\label{eisenhart_equ}
\partial_{xy}\varphi + \partial_y \ln H_1 \partial_x \varphi + \partial_x \ln H_2 \partial_y \varphi + \varphi \partial_{xy} \ln (H_1 H_2)=0
\end{equation}
gives rise to a pair of functions $h$ and $l$ described by
\begin{equation}\label{comb_from_phi}
\begin{aligned}
\partial_x h &= \phantom{-} \varphi \partial_x \ln (H_2\varphi),
\\\partial_y h &= -\varphi \partial_y \ln H_1,
\\l &= h -\varphi,
\end{aligned} 
\end{equation}
that fulfills equations (\ref{comp_comb}) and therefore defines by (\ref{def_comb}) a Combescure transformation of $f$. 

Since the functions $h$ and $l$ are uniquely determined up to the same additive constant, a function $\varphi$ satisfying condition (\ref{eisenhart_equ}) induces by (\ref{comb_from_phi}) a 1-parameter family of Combescure transforms of $f$.
\\\\Note that, in particular, the normals of $f$ and a Combescure transform $\hat{f}$ coincide up to orientation, that is, $n=\delta\hat{n} $, where $\delta \in \{ \pm 1 \}$ depending on whether the product $hl$ is positive or negative. Hence, denoting the principal curvatures of $f$ by $\kappa_1$ and $\kappa_2$, the principal curvatures $\hat{\kappa}_1$ and $\hat{\kappa}_2$ of the Combescure transform $\hat{f}$ are given by 
\begin{equation*}
\hat{\kappa}_1=\delta\frac{\kappa_1}{h} \ \ \ \text{and } \ \ \ \hat{\kappa}_2=\delta\frac{\kappa_2}{l}.
\end{equation*}
We also allow the functions $h$ or $l$ to vanish identically, such that $\hat{f}$ degenerates to a curve or a point. In these cases we consider the corresponding radii of principal curvature of $\hat{f}$ to be zero.
%
%
%
\subsection{Triply orthogonal systems and their Combescure transforms}\label{section_tos}
In what follows, we give a brief introduction to triply orthogonal systems and summarize some facts that will become important later. We then focus on the Combescure transformation for triply orthogonal systems and prove two results that show how these transforms can be constructed. 

For more details and a wide range of beautiful results in this area, we refer to the rich literature on triply orthogonal systems, for example \cite{bianchi_lezioni,Darboux_ortho, cyclic_guichard_eth, salkowski, weatherburn}.
\\\\\textit{Agreement.} Throughout the paper we use the following convention for the indices in the realm of triply orthogonal systems: if only one index $i$ or two indices $i$ and $j$ appear in an equation, then $i \neq j \in \{1,2,3 \}$. If there are three indices $(i, j, k)$ involved, then they take the values $(1,2,3)$ or a cyclic permutation of them. 

To avoid useless indices, we sometimes change the notations of variables: $x=x_1$, $x_2=y$ and $x_3=z$. For the partial derivatives we use the abbreviations $\partial_i=\partial_{x_i}$ and $\partial_{ij}=\partial_{x_i}\partial_{x_j}$.

\bigskip 

Let $U$ be an open and connected subset of $\mathbb{R}^3$, then $f:U\rightarrow \mathbb{R}^3$ is a \emph{triply orthogonal system} if $\text{det}(\partial_x f, \partial_y f, \partial_z f)\neq 0$ and
\begin{equation*}
(\partial_i f, \partial_j f)=0,
\end{equation*}
where $(.,.)$ denotes the usual inner product on $\mathbb{R}^3$. The unit normals $N_i$ of the coordinate surfaces $x_i=const$ are given by 
\begin{equation*}
N_i:=\frac{\partial_j f \times \partial_k f}{|\partial_j f \times \partial_k f|}
\end{equation*}
and the functions $H_1, H_2, H_3: U \rightarrow \mathbb{R}$ defined by
\begin{equation*}
\partial_i f = H_i N_i
\end{equation*}
are called the \emph{Lam\'{e} coefficients} of $f$. Since 
\begin{equation*}
\text{det} (\partial_x f, \partial_y f, \partial_z f)= H_1 H_2 H_3,
\end{equation*}
the Lam\'{e} coefficients never vanish on $U$. The induced metric of a triply orthogonal system is then given by
\begin{equation}\label{induced_metric}
I=(df,df)=H_1^2 dx^2+H_2^2 dy^2+H_3^2 dz^2.
\end{equation}

\bigskip

If $f$ is a triply orthogonal system, then the functions $(f_1,f_2,f_3)=f$ are solutions of the point equations of any surface family $x_k=const$, 
\begin{align}\label{point_TOS}
\partial_{ij} f = \partial_j \ln H_i \partial_i f + \partial_i \ln H_j \partial_j f,
\end{align}
and fulfill the Gauss equation
\begin{equation*}
\partial_{i}^2f= \partial_i \ln H_i \partial_i f - \frac{H_i}{H_j^2}\partial_j H_i \partial_j f - \frac{H_i}{H_k^2}\partial_k H_i \partial_k f.
\end{equation*}
\bigskip 

Moreover, the Lam\'{e} coefficients $H_1, H_2$ and $H_3$ satisfy the first and second system of \emph{Lam\'{e}'s equations}
\begin{align}
\partial_{ij} H_k &= \partial_j \ln H_i \partial_i H_k + \partial_i \ln H_j \partial_j H_k,\label{lame_first}
\\0&=\frac{1}{H_k^2}\partial_k H_i \partial_k H_j + \partial_j \Big\{ \frac{1}{H_j} \partial_j H_i\Big\} + \partial_i \Big\{ \frac{1}{H_i} \partial_i H_j \Big\}.\label{lame_second}
\end{align}
If we introduce Darboux's \emph{rotational coefficients}
\begin{equation*}
\beta_{ij}:=\frac{1}{H_i}\partial_i H_j,
\end{equation*}
then Lam\'{e}'s systems (\ref{lame_first}) and (\ref{lame_second}) reduce to first order equations
\begin{align*}
0&=\partial_k \beta_{ij}- \beta_{ik}\beta_{kj},
\\0&=\partial_i \beta_{ij} + \partial_j \beta_{ji} + \beta_{ki}\beta_{kj}
\end{align*}
and the unit normals $N_1, N_2$ and $N_3$ satisfy the system
\begin{equation}\label{direction_cos}
\begin{aligned}
\partial_j N_i &= \beta_{ij} N_j,
\\\partial_i N_i &= - \beta_{ji} N_j - \beta_{ki} N_k,
\\(N_i,N_j)&=\delta_{ij}.
\end{aligned}
\end{equation}
Conversely, an orthogonal metric of the form (\ref{induced_metric}) that satisfies Lam\'{e}'s equations gives rise to a triply orthogonal system, which is uniquely determined up to Euclidean motions: to recover a parametrization $f:U \rightarrow \mathbb{R}^3$ from a prescribed orthogonal metric $I$, we first compute the rotational coefficients $\beta_{ij}$ and construct a set of suitable normals $N_1, N_2$ and $N_3$ from system (\ref{direction_cos}). Then a parametrization $f:U \rightarrow \mathbb{R}^3$ of the system with induced metric $I$ is obtained by integrating
\begin{equation}\label{recover_f}
df= H_1 N_1 dx + H_2 N_2 dy + H_3 N_3 dz. 
\end{equation}  
Note that the integrability of the system (\ref{direction_cos}) and equation (\ref{recover_f}) is guaranteed by Lam\'{e}'s equations and the construction is unique up to Euclidean motions.

Thus, if the ambiguity of Euclidean motions does not effect the geometry we are interested in, we talk about the induced metric instead of the explicit parametrization.

\bigskip

Recall that, by Dupin's Theorem, any two coordinate surfaces of two different families intersect along a curve which is a curvature line for both surfaces. This immediately reveals information about the geometry of the three families of coordinate surfaces: the induced metrics $I_i$ for the family of coordinate surfaces $x_i=const$ are given by
\begin{equation*}
I_i=H_j^2 dx_j^2+H_k^2 dx_k^2
\end{equation*}
and therefore the two principal curvatures of these coordinate surfaces read
\begin{equation*}
\kappa_{ij}=-\frac{\beta_{ij}}{H_j} \ \ \ \text{and} \ \ \ \kappa_{ik}=-\frac{\beta_{ik}}{H_k}.
\end{equation*} 
Hence, due to Lam\'{e}'s equations (\ref{lame_first}), the metric coefficient $H_i|_{x_i=const}$ is a solution of the point equation (\ref{point_TOS}) of the coordinate surface $x_i=const$.
%
%
\\\\Remarkable subclasses of triply orthogonal systems emerge from induced metrics whose traces fulfill a certain condition (cf.\,\cite{bianchi_1918, Darboux_ortho, Guichard_small}). We use the following terminology: 
%
%
\begin{defi}[{\cite[\S 64]{Guichard_small}}]
A \emph{$\chi$-system} is a triply orthogonal system such that the trace with respect to the Minkowski-metric of the induced metric $I$ satisfies
\begin{equation}\label{trace_chi}
H_1^2+H_2^2-H_3^3=\chi.
\end{equation}
In particular, a $0$-system is called a \emph{Guichard net}.
\end{defi}
Another distinguished subclass in this realm consists of $\alpha^2|f|^2$-systems, where $\alpha \in \mathbb{R}$ (see \cite[Chap.\,X]{Darboux_ortho}). As well as Guichard nets, this class is invariant under inversions: let $(f, H_1,H_2, H_3)$ be an $\alpha^2|f|^2$-system, then
\begin{align*}
f &\rightarrow f':=\frac{f}{|f|^2}
\\H_i^2 &\rightarrow {H'_i}^{2}:=\frac{H_i^2}{|f|^4} 
\end{align*}
and, therefore, $f'$ is an $\alpha^2|f'|^2$-system.
\\\\Better insights into the geometry of a triply orthogonal system is often obtained by studying the coordinate surfaces. Thus, many geometric ideas known from surface theory were classically carried over to triply orthogonal systems by requiring a geometric property for any coordinate surface of the system: for example, systems consisting of isothermic surfaces or of constant Gaussian curvature surfaces. Also the rich transformation theory for triply orthogonal systems is based on this concept. 

The Combescure transformation between two triply orthogonal systems is named after its first appearance in \cite{comb}:

\begin{defi} \label{def_combescure}
Two triply orthogonal systems $f, \hat{f}: U \rightarrow \mathbb{R}^3$ are related by a \emph{Combescure transformation} if any two corresponding coordinate surfaces of $f$ and $\hat{f}$ are Combescure transforms of each other.
\end{defi}
\noindent Since the rotational coefficients $\beta_{ij}$ determine the normal directions of the coordinate surfaces of a triply orthogonal system, two systems are related by a Combescure transformation if and only if they share the same rotational coefficients (cf.\,\cite[\S 5]{salkowski}).

The following lemma gives a criterion for three functions to define a Combescure transformation of a given parametrized triply orthogonal system:
%
%
\begin{prop}\label{combescure}
Let $f$ be a triply orthogonal system with Lam\'{e} coefficients $(H_i)_i$, then 
\begin{equation}\label{equ_comb}
d\hat{f} = h_1 \partial_x f dx + h_2 \partial_y f dy + h_3 \partial_z f dz 
\end{equation}
defines a Combescure transform $\hat{f}$ of $f$ if and only if
\begin{equation}
\partial_i h_j = (h_i-h_j)\partial_i \ln H_j.\label{comb_cond}
\end{equation}
\end{prop}
\begin{proof}
Let us denote the rotational coefficients of $\hat{f}$ by $\hat{\beta}_{ij}:=\frac{1}{h_iH_i}\partial_i (h_j H_j)$, then
\begin{align*}
\beta_{ij}-\hat{\beta}_{ij}=\frac{H_j}{h_iH_i} \Big\{ (h_i-h_j)\partial_i \ln H_j - \partial_i h_j \Big\}.
\end{align*}
Hence, $\beta_{ij}=\hat{\beta}_{ij}$ if and only if the equations (\ref{comb_cond}) hold, which proves the claim.
\end{proof}
Thus, for two given metrics $I=\sum_{i=1}^3 H_i^2 dx_i^2$ and $\hat{I}=\sum_{i=1}^3 (h_i H_i)^2 dx_i^2$, where the functions $h_1, h_2$ and $h_3$ satisfy equations (\ref{comb_cond}), we can find two parametrizations $f$ and $\hat{f}$ which have these metrics as induced metrics and are related by a Combescure transformation. Therefore, we also say that two such induced metrics are related by a \emph{Combescure transformation}.
%
%
\\\\To conclude this section, we prove a rather technical result, which will become important later. Suppose we start with a triply orthogonal system and Combescure transform each coordinate surface. Then the following proposition gives a criterion in which cases the Combescure transformed surfaces constitute again a triply orthogonal system:
\begin{prop}\label{tos_comb_technical}
Let $(f, H_1, H_2, H_3)$ be a triply orthogonal system and let $\varphi_1, \varphi_2$ and $\varphi_3$ denote three functions that define by (\ref{comb_from_phi}) Combescure transformations for the corresponding family of coordinate surfaces $x_i =const$. Then, there exists a Combescure transformed triply orthogonal system $\hat{f}$ of $f$ such that the Combescure transformations between the coordinate surfaces are induced by the corresponding functions $\varphi_1$, $\varphi_2$ and $\varphi_3$ if and only if
\begin{equation}\label{tos_comb_phi}
\varphi_1+ \varphi_2+\varphi_3=0 \ \ \text{and} \ \ \partial_j \varphi_j = \varphi_i \partial_j \ln H_k + \varphi_k \partial_j \ln H_i.
\end{equation}
\end{prop}
\begin{proof}
By Proposition \ref{combescure}, any Combescure transform $\hat{f}$ of $f$ can be described by three functions $h_1, h_2$ and $h_3$ satisfying (\ref{comb_cond}). On the other hand, on the level of coordinate surfaces, any function $\varphi_j$ gives rise to functions $h_i^j$ and $h_k^j$ fulfilling (cf.\,conditions (\ref{comb_from_phi}))
\begin{equation*}
\begin{aligned}
\partial_k h_i^j&=\phantom{-}\varphi_j \partial_k \ln H_i
\\\partial_i h_i^j&=-\varphi_j \partial_i \ln (H_k \varphi_j)
\\\partial_k h_k^j&=\phantom{-}\varphi_j \partial_k \ln (H_i \varphi_j)
\\\partial_i h_k^j&=-\varphi_j \partial_i \ln H_k.
\end{aligned}
\end{equation*}
Therefore, we obtain a Combescure transformed triply orthogonal system $\hat{f}$ of $f$ with the required property if and only if $h_j^i=h_j^k$.

Thus, if we assume that the conditions (\ref{tos_comb_phi}) are satisfied, then 
\begin{equation}\label{comb_equ_phi}
\partial_j h_j^i=\varphi_i \partial_j \ln(H_k \varphi_i)=-\varphi_k \partial_j\ln(H_i\varphi_k)=\partial_jh_j^k
\end{equation} 
hold. Hence, it remains to show that the system
\begin{equation}\label{system_comb_technical}
\begin{aligned}
\partial_j h_j &=\phantom{-}\varphi_i \partial_j \ln(H_k \varphi_i)
\\\partial_k h_j &= - \varphi_i\partial_k \ln H_j
\\\partial_i h _j &= \phantom{-}\varphi_k \partial_i \ln H_j
\end{aligned}
\end{equation}
is indeed integrable. Since $\varphi_i$ and $\varphi_k$ induce Combescure transformations, we obtain by equations (\ref{comb_from_phi}) the integrability conditions 
\begin{equation*}
\partial_j(\partial_k h_j)=\partial_k(\partial_j h_j) \ \ \text{and} \ \ \partial_i(\partial_j h_j)=\partial_j(\partial_i h_j).
\end{equation*}
The remaining integrability condition $\partial_i(\partial_k h_j)=\partial_k(\partial_i h_j)$ then follows from the assumptions (\ref{tos_comb_phi}) and the first system of Lam\'e's equations of $f$.

Conversely, suppose that $h_j^i=h_j^k$, then equation (\ref{comb_equ_phi}) is satisfied
and we conclude that the conditions (\ref{comb_from_phi}) hold.
\end{proof}
%
%
%
\section{G-surfaces and G-systems}\label{section_G_surfaces}
\subsection{G-surfaces}
We present, compared to the definition given by Calapso in \cite[\S VI]{calapso_G}), a slightly generalized definition of G-surfaces, where we allow for  three different subclasses depending on the choice of $\varepsilon \in \{0, \pm 1\}$:
\begin{defi}
A surface $f: M^2 \rightarrow \mathbb{R}^3$ is called a \emph{G-surface} if there exist curvature line coordinates with induced metric $I=H_1^2 dx^2+H_2^2dy^2$ and a Combescure transform $\hat{f}$ of $f$ such that
\begin{equation}\label{cond_G}
cH_1^2H_2^2\Big(\frac{\kappa_1}{\hat{\kappa}_1}-\frac{\kappa_2}{\hat{\kappa}_2}\Big)^2=H_2^2 + \varepsilon H_1^2
\end{equation}
for some $\varepsilon \in \{ 0 , \pm 1 \}$, $c \in \mathbb{R}\setminus \{ 0 \}$ and $\kappa_1, \kappa_2, \hat{\kappa}_1$ and $\hat{\kappa}_2$ the principal curvatures of $f$ and $\hat{f}$, respectively.

The surface $\hat{f}$ will be called an \emph{associated surface of $f$}.
\end{defi} 
Obviously, the \emph{G-condition} (\ref{cond_G}) depends on the choice of particular curvature line coordinates. However, from the definition we immediately deduce that a surface given in (arbitrary) curvature line coordinates $(x,y)$ is a G-surface if and only if there exist functions $\chi_1=\chi_1(x)$ and $\chi_2=\chi_2(y)$ such that
\begin{equation}\label{G_repara}
c\Big(\frac{\kappa_1}{\hat{\kappa}_1}-\frac{\kappa_2}{\hat{\kappa}_2}\Big)^2=\frac{1}{\chi_1^2 H_1^2} + \frac{\varepsilon}{\chi_2^2 H_2^2}.
\end{equation}
\ \\\\As already observed in \cite[\S VI (89)]{calapso_G}), a G-surface admits a special solution of its point equation:
\begin{prop}\label{G_point_invariant}
A surface $f:M^2 \rightarrow \mathbb{R}^3$ is a G-surface if and only if there exist curvature line coordinates such that the function $\sqrt{H_2^2 + \varepsilon H_1^2}$ is a scalar solution of its point equation (\ref{point_equ_surf}).
\end{prop}
\begin{proof}
By definition a surface is a G-surface if and only if there exist curvature line coordinates and a Combescure transformation determined by functions $h$ and $l$ such that
\begin{equation*}
(h-l)^2 = \frac{H_2^2+\varepsilon H_1^2}{cH_1^2H_2^2}. 
\end{equation*}
Hence, this is the case if and only if the function $\varphi:= \frac{\sqrt{H_2^2+\varepsilon H_1^2}}{\sqrt{c}H_1H_2}$ fulfills equation (\ref{eisenhart_equ}). A straightforward computation shows that equation (\ref{eisenhart_equ}) holds if and only if  $\sqrt{H_2^2 + \varepsilon H_1^2}$ is a solution of the point equation of $f$.
\end{proof}
\noindent As a consequence of the construction in this proof, the function $\varphi$ gives rise to a 1-parameter family of Combescure transformations of $f$ such that the G-condition (\ref{cond_G}) is satisfied. Hence, any G-surface admits (up to orientation) a 1-parameter family of associated surfaces: let $\hat{f}_0$ be an arbitrary associated surface of $f$, then this 1-parameter family $(\hat{f}_c)_{c \in \mathbb{R}}$ is given by
\begin{equation*}
\hat{f}_c:= \hat{f}_0+cf.
\end{equation*}
In particular, any three associated surfaces $\hat{f}_1$, $\hat{f}_2$ and $\hat{f}_3$ satisfy the relation
\begin{equation*}
(\hat{f}_1-\hat{f}_3)=\lambda (\hat{f}_2 - \hat{f}_3)
\end{equation*}
for a constant $\lambda \in \mathbb{R}\setminus\{ 0 \}$.

\bigskip

\noindent As many integrable surface classes, we will prove that also G-surfaces can be characterized by the existence of Combescure transforms fulfilling a special relation between their principal curvatures:
\begin{thm}\label{thm_G_dual}
A surface $f$ is a G-surface if and only if there exists a non-trivial pair $(\hat{f}, f^\star)$ of Combescure transforms of $f$ with principal curvatures satisfying
\begin{equation}\label{G_cond_curv}
\frac{1}{\kappa_1\kappa_2^\star}+\frac{1}{\kappa_2 \kappa_1^\star}=-\frac{2}{\hat{\kappa}_1\hat{\kappa}_2}.
\end{equation}
A Combescure related pair $(\hat{f}, f^\star)$ of $f$ will be called \emph{non-trivial} if 
\begin{equation}\label{non_trivial}
\kappa_1\kappa_1^\star \neq - \hat{\kappa}_1^2 \ \ \text{or} \ \ \kappa_2\kappa_2^\star \neq -\hat{\kappa}_2^2.
\end{equation}
\end{thm}
\begin{rem}\label{rem_trivial} Note that any arbitrary surface $f$ admits a pair of trivial Combescure transforms such that the relation (\ref{G_cond_curv}) holds: namely, for any $\lambda \in \mathbb{R}\setminus \{ 0 \}$, the choice
\begin{equation*}
\hat{f}:=\lambda f \ \ \ \text{and } \ \ \ f^\star:=-\lambda^2 f
\end{equation*}
provides two such Combescure transforms of $f$. 
\end{rem}
%
%
\noindent\textit{Proof of Theorem \ref{thm_G_dual}.}
Suppose that $f$ is a G-surface with induced metric 
\begin{equation*}
I=H_1^2dx^2+H_2^2dy^2
\end{equation*}
and let $\hat{f}$ be an associated surface defined by
\begin{equation*}
\partial_x \hat{f} = h \partial_x f \ \ \text{and} \ \  \partial_y \hat{f} = l \partial_y f
\end{equation*}
with principal curvatures $\hat{\kappa}_1$ and $\hat{\kappa}_2$.

We will prove that the sought-after Combescure transform $f^\star$ of $f$ is, for a suitable choice of $\tilde{\delta} \in \{ \pm 1 \}$, obtained by integrating
\begin{equation}\label{def_fstar}
df^\star=\tilde{\delta}\big( -h^2+\frac{1}{H_1^2} \big)\partial_x f \ dx + \tilde{\delta}\{ -l^2+\frac{\varepsilon}{H_2^2} \}\partial_y f \ dy.
\end{equation}
Indeed, a straightforward computation using the G-condition (\ref{cond_G}) shows that (\ref{def_fstar}) is integrable. Since the principal curvatures of $f^\star$ are given by
\begin{equation*}
\kappa_1^\star = \tilde{\delta}\delta\frac{H_1^2\kappa_1}{-h^2H_1^2+1} \ \ \text{and} \ \ \kappa_2^\star = \tilde{\delta}\delta\frac{H_2^2\kappa_2}{-l^2H_2^2+\varepsilon},
\end{equation*}
where $\delta \in \{ \pm 1 \}$ depends on the orientation of the normal of $f^\star$, we obtain that
\begin{equation*}
\tilde{\delta}\delta\Big( \frac{1}{\kappa_1\kappa_2^\star}+\frac{1}{\kappa_2\kappa_1^\star}\Big )= -\frac{2hl}{\kappa_1\kappa_2}=-\frac{2}{\hat{\kappa}_1\hat{\kappa}_2}.
\end{equation*}
Hence, for the choice $\tilde{\delta}=\delta$, we have constructed the sought-after pair $(\hat{f}, f^\star)$, which is non-trivial.
%
%
\\\\Conversely, suppose that $\hat{f}$ and $f^\star$ are non-trivial Combescure transforms of $f$ which satisfy condition ($\ref{G_cond_curv}$). Then, there exist functions $h', l'$ and $h'', l''$ such that 
\begin{align*}
H_1&=h'\hat{H}_1, \ \ \ \ H_2=l'\hat{H}_2,
\\H_1^\star&=h''\hat{H}_1, \ \ \ H_2^\star=l''\hat{H}_2
\end{align*}
and the principal curvatures of $f$ and $f^\star$ read $\kappa_1=\hat{\kappa}_1/h'$ and $\kappa_1^\star=\hat{\kappa}_1/h''$, respectively.  From equation (\ref{G_cond_curv}), we then deduce that 
\begin{equation}\label{cond_hs}
h'l''+l'h''=-2
\end{equation}
holds.

Moreover, from the compatibility conditions (\ref{comp_comb}) and  (\ref{cond_hs}) we obtain
\begin{align*}
\partial_y(h'h'')&=-2\{1+h'h''\}\partial_y \ln \hat{H}_1,
\\\partial_x(l'l'')&=-2 \{1+l'l'' \}\partial_x \ln \hat{H}_2
\end{align*}
and therefore
\begin{align*}
h'h''&=-1+\frac{\psi_1(x)}{\hat{H}_1^2},
\\l'l''&=-1+\frac{\psi_2(y)}{\hat{H}_2^2}
\end{align*}
for suitable functions $\psi_1(x)$ and $\psi_2(y)$. Since, by (\ref{non_trivial}), we have
\begin{equation*}
h'h'' \neq 1 \ \ \text{or} \ \ l' l'' \neq 1,
\end{equation*}
at least one of the functions $\psi_1(x)$ and $\psi_2(y)$ does not vanish.
\\\\Furthermore, substituting in equation (\ref{cond_hs}), gives
\begin{equation*}
\Big( \frac{\kappa_1}{\hat{\kappa}_1}-\frac{\kappa_2}{\hat{\kappa}_2}\Big)^2=\frac{\psi_2(y)}{H_2^2}+\frac{\psi_1(x)}{H_1^2},
\end{equation*}
which leads, for a suitable choice of $\varepsilon \in \{0, \pm 1 \}$ and $c \in \mathbb{R}\setminus \{ 0 \}$ to the G-condition (\ref{G_repara}). Thus, $f$ is a G-surface. \qed
\\\\Therefore, due to symmetry, it follows that $\hat{f}$ is also an associated surface of $f^\star$ and we obtain another G-surface:
\begin{corand}
Let $f$ be a G-surface, then the Combescure transform $f^\star$ defined by (\ref{G_cond_curv}) is a G-surface, which will be called the \emph{dual surface of $f$ with respect to the associated surface $\hat{f}$}.
\end{corand}
\begin{center}
\begin{tikzcd}
&[-2.5em] &[-1em] \bm{f} \arrow[dl,swap,dash,"\mathcal{C}"] \arrow[d,swap,dash,"\mathcal{C}"] \arrow[rd,dash,"\mathcal{C}"] &[-1em]&[-2.5em]\\
\cdots & \bm{\hat{f}_{c-1}} \arrow[d,dash,swap,"\mathcal{C}"]& \bm{\hat{f}_{c}}\arrow[d,swap,dash,"\mathcal{C}"] & \bm{\hat{f}_{c+1}}\arrow[d,dash,swap,"\mathcal{C}"] & \cdots\\
\cdots & \bm{f^\star_{c-1}} & \bm{f^\star_{c}} & \bm{f^\star_{c+1}} & \cdots
\end{tikzcd}
\\[6pt]\small\emph{Figure 1.} A Guichard net $f$ with a 1-parameter family of Combescure related associated systems $(\hat{f}_c)_{c \in  \mathbb{R}}$ and the corresponding dual systems $(f^\star_c)_{c \in  \mathbb{R}}$.
\end{center}
\ \\Note that a dual surface depends on the choice of the associated surface: let $\hat{f}_0$ be an associated surface with corresponding dual surface $f^{\star}_0$, then the dual surface $f^{\star}_c$ with respect to the associated surface $\hat{f}_c$ is given by
\begin{equation*}
f_c^\star:=f_0^\star-2c\hat{f}_0-c^2f=f_0^\star-2c\hat{f}_{\frac{c}{2}}.
\end{equation*}
Furthermore, let us fix some dual surface $f_{\bar{c}}^\star$ of $f$. Then the 1-parameter family of associated surfaces of $f_{\bar{c}}^\star$ becomes
\begin{equation*}
\hat{f}_{\bar{c}d}:=\hat{f}_{\bar{c}}+df_{\bar{c}}^\star, \ \ d \in \mathbb{R},
\end{equation*}
while the corresponding dual surfaces of $f_{\bar{c}}^\star$ are given by
\begin{equation*}
(f_{\bar{c}}^{\star})^\star_d:=f-2d\hat{f}_{\bar{c}}-d^2f_{\bar{c}}^\star.
\end{equation*}
Hence, we recover the original surface $f$ in the dual family of any dual surface $f_{\bar{c}}^\star$ of $f$ at $d=0$,
\begin{equation*}
(f_{\bar{c}}^\star)_0^\star = f,
\end{equation*} 
and the described construction indeed provides a duality.
\\\\We further remark that, generically, the two associated families $(\hat{f}_c)_{c \in \mathbb{R}}$ and $(\hat{f}_{\bar{c}d})_{d\in\mathbb{R}}$ do not coincide, while for $d\neq 0$ the dual surfaces $(f_{\bar{c}}^{\star})^\star_{d\in\mathbb{R}}$  amount to scalings of the family $(f_c^\star)_{c \in \mathbb{R}}$ (cf.\,Example \ref{example_guichard} in Section \ref{section_guichard_nets}).
\ \\\\ \ 
\begin{center}
\includegraphics[scale=0.35]{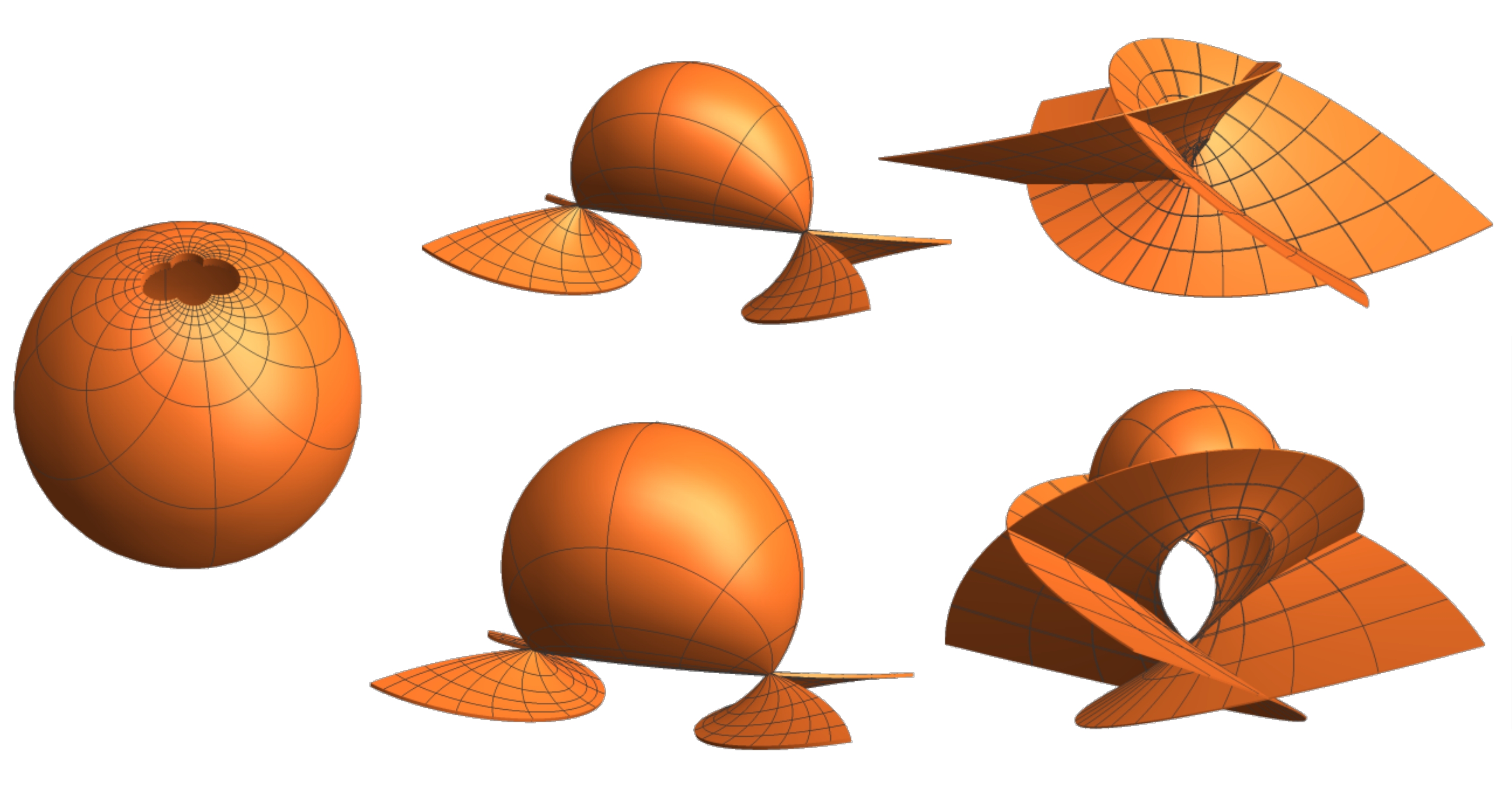}
\\[6pt]\small\emph{Figure 2.} A sphere with two Dupin cyclides as associated surfaces (\textit{middle}) and the two corresponding dual surfaces (\textit{right}).
\end{center}
\bigskip

Apart from providing new G-surfaces from a given one, the characterization discussed in Theorem~\ref{thm_G_dual} shows that G-surfaces are contained in the integrable class of  O-surfaces \cite{O_surface}: 
if we consider a non-trivial triplet of Combescure transformations and equip the dual space with the indefinite metric induced by the matrix
\begin{equation}\label{matrix_O}
\begin{pmatrix}
0&1&0\\
1&0&0\\
0&0&2
\end{pmatrix},
\end{equation}
then the orthogonality condition that defines the corresponding subclass of O-surfaces, is equivalent to the condition (\ref{G_cond_curv}).

However, note that the requirement of non-trivial Combesure transforms is indeed necessary: by Remark \ref{rem_trivial}, any surface satisfies the criterion of an O-surface with respect to the indefinite metric induced by (\ref{matrix_O}), if trivial Combescure transformations are not excluded.

\subsection{Examples} Using the just derived  characterizations we demonstrate how various well-known surface classes arise in the context of G-surfaces. 
\bigskip 

\noindent\textit{Isothermic surfaces.} Since any isothermic surface admits conformal curvature line coordinates, any scaling of $f$ provides an associated surface as seen from the definition using $\varepsilon=-1$. The converse also holds: 
\begin{prop}\label{prop_G_isothermic}
A G-surface $f$ is isothermic if and only if the family of associated surfaces $\hat{f}_c$ is given by scalings of $f$, that is,  $\hat{f}_c=cf$ for all $c\in \mathbb{R}$.
\end{prop}
\noindent As a consequence, in the context of G-surfaces, the dual surfaces of an isothermic surface $f$ with induced metric $I=e^{2\psi}\{ dx^2+dy^2 \}$ are determined by
\begin{equation*}
df_c^\star=\{ -c^2+e^{-2\psi} \} \partial_x f dx + \{ -c^2-e^{-2\psi}\} \partial_y f dy.
\end{equation*}
Hence, if for $c=0$ the associated surface degenerates to a point, the corresponding dual surface becomes the classical Christoffel dual \cite{christoffel} of an isothermic surface and therefore it is again isothermic. However, the other dual surfaces $f_c^\star$ are in general not isothermic again.
\noindent\\\\\textit{Guichard surfaces.} Recall \cite{calapso_guichard} that a surface is a Guichard surface if and only if it fulfills Calapso's equation
\begin{equation*}
c(\kappa_1-\kappa_2)^2H_1^2H_2^2=H_2^2+\varepsilon H_1^2
\end{equation*}
for suitable curvature line coordinates, $\varepsilon \in \{\pm 1, 0\}$ and some $c \in \mathbb{R}\setminus \{ 0 \}$. 
\\\\Hence, since the spherical representative of a surface is a Combescure transform  that is determined by the functions $h=\kappa_1$ and $l=\kappa_2$, any Guichard surface is a G-surface by Calapso's equation. 

Moreover, the spherical representative as an associated surface distinguishes this subclass:
\begin{prop}\label{prop_G_Guichard}
A G-surface $f$ is a Guichard surface if and only if its spherical representative is an associated surface of $f$.
\end{prop}
\noindent We remark that, in the case of a Guichard surface, condition (\ref{G_cond_curv}) becomes the relation originally used to define Guichard surfaces \cite{guichard_defi}, namely
\begin{equation*}
\frac{1}{\kappa_1\kappa_2^\star}+\frac{1}{\kappa_2\kappa_1^\star}=-2.
\end{equation*}
Therefore, the dual surface $f^\star$ with respect to the spherical representative, constructed in the context of G-surfaces, coincides with the classical dual surface given by Guichard.
\\\\This subclass also provides examples of G-surfaces that satisfy the G-condition for $\varepsilon=0$: in this case, we deduce from the G-condition and the compatibility conditions (\ref{comp_comb}) that the associated surfaces are determined by functions $h$ and $l$ of the form
\begin{align}
h&=\pm\frac{1}{H_1^2}+\psi_1
\\l&=\psi_1,
\end{align}
where $\psi_1=\psi_1(x)$. 

Thus, for a Guichard surface, where one of the associated surfaces is given by the spherical representative, one of the principal curvatures of $f$ is constant along the corresponding curvature direction. Therefore, $f$ is a channel surface (cf.\,\cite{calapso_guichard}).
%
%
%
%
\\\\\textit{$\Omega$-surfaces.} By Demoulin's equation \cite{demoulin_associated}, a surface is an $\Omega$-surface if and only if there exists a Combescure transformed dual surface $f^\star$ fulfilling
\begin{equation}\label{demoulin}
\Big(\frac{1}{\kappa_1}-\frac{1}{\kappa_2}\Big)\Big(\frac{1}{\kappa_1^\star}-\frac{1}{\kappa_2^\star}\Big)=\frac{1}{\kappa_1^2H_1^2}+\frac{\varepsilon}{\kappa_2^2H_2^2},
\end{equation} 
for a suitable choice of curvature line coordinates inducing the first fundamental form $I=H_1^2dx^2+H_2^2dy^2$.
\\\\A reformulation of the condition (\ref{G_cond_curv}) between the principal curvatures of a G-surface and its associated and dual surfaces, will shed light on the relation between the dual surfaces of a G-surface and of an $\Omega$-surface:
\begin{lem}
Let $f$ be a G-surface, then the principal curvatures of an associated surface $\hat{f}$ and the corresponding dual surface $f^\star$ satisfy
\begin{equation*}
\Big( \frac{1}{\kappa_1}-\frac{1}{\kappa_2} \Big)\Big(\frac{1}{\kappa_1^\star}-\frac{1}{\kappa_2^\star} \Big)= \frac{1}{\kappa_1^2H_1^2}+\frac{\varepsilon}{\kappa_2^2H_2^2}-\Big(\frac{1}{\hat{\kappa}_1}-\frac{1}{\hat{\kappa}_2} \Big)^2.
\end{equation*}
\end{lem}
\begin{proof}
Suppose $f$ is a G-surface. Then the principal curvatures of a dual surface with respect to the associated surface
\begin{equation*}
d\hat{f}=h\partial_x f dx + l \partial_y f dy
\end{equation*} 
are given by
\begin{align*}
\kappa_1^\star&=\frac{\kappa_1H_1^2}{-h^2H_1^2+1},
\\\kappa_2^\star&=\frac{\kappa_2H_2^2}{-l^2H_2^2+\varepsilon}.
\end{align*}
Now a straightforward computation using the G-condition (\ref{cond_G}) proves the claim.
\end{proof}
Therefore, a dual surface of a G-surface fulfills Demoulin's equation (\ref{demoulin}) if and only if 
\begin{equation*}
l\kappa_1-h\kappa_2=0.
\end{equation*}
Thus, we have proven:
\begin{prop}
The dual surface of a $G$-surface provides an $\Omega$-dual surface if and only if the corresponding associated surface is totally umbilic or degenerates to a point.
\end{prop}
In particular, this proposition reflects that isothermic and Guichard surfaces belong to the classes of $G$- as well as that of $\Omega$-surfaces. In these cases, the corresponding associated surface coincides with the associated Gauss map for $\Omega-$surfaces, which was recently determined in \cite{pember_lie_applicable}.
%
%
\subsection{G-systems}
Although it is quite popular to distinguish surface classes by the existence of particular Combescure transforms, there are only few results in this direction in the realm of triply orthogonal systems (for example, the Combescure transforms of the confocal quadrics \cite{salkowski}). 

However, it is a prominent question to analyse systems consisting of particular coordinate surfaces: for example, systems with isothermic coordinate surfaces \cite[Chap.\,III--V]{Darboux_ortho} or Bianchi's systems consisting of coordinate surfaces with constant Gaussian curvature (cf.\,\cite{bianchi_type_guichard, salkowski}). Clearly, these systems are all built from G-surfaces (cf.\,Propositions \ref{prop_G_isothermic} and \ref{prop_G_Guichard}).
\\\\Here, we are interested in both aspects and investigate triply orthogonal systems consisting of G-surfaces that admit Combescure transforms constituted by the associated and dual surfaces. In particular, it will turn out that Guichard nets belong to this class of triply orthogonal systems.
\begin{defi}
A triply orthogonal system $f$ that admits non-trivial Combescure transforms $\hat{f}$ and $f^\star$ such that the principal curvatures of the coordinate surfaces satisfy
\begin{equation}\label{cond_G_system}
\frac{1}{\kappa_{ij}\kappa_{ik}^\star}+\frac{1}{\kappa_{ik}\kappa_{ij}^\star}=-\frac{2}{\hat{\kappa}_{ij}\hat{\kappa}_{ik}}
\end{equation}
will be called a \emph{G-system}. The Combescure transforms $\hat{f}$ and $f^\star$ are said to be an \emph{associated system} and a \emph{dual system}, respectively.
\end{defi}
\noindent If we describe the Combescure transform $\hat{f}$ of a G-system $f$ by
\begin{equation}\label{cond_asso_f}
d \hat{f} = h_1 \partial_x f dx + h_2 \partial_y f dy + h_3 \partial_z f dz, 
\end{equation}
we deduce, using similar arguments as in the proof of Theorem~\ref{thm_G_dual} that there exist functions $\psi_i=\psi_i(x_i)$ such that
\begin{equation*}
(h_j - h_k)^2 = \frac{1}{\psi_k H_k^2}+\frac{1}{\psi_j H_j^2}.
\end{equation*}
Hence, the coordinate surfaces of a G-system are G-surfaces and $\hat{f}$ consists of the corresponding associated surfaces.

The corresponding dual system $f^\star$ is then given by  
\begin{equation}\label{star_from_assoc}
df^\star = h_1^\star \ \partial_x f dx + h_2^\star \ \partial_y f dy + h_3^\star \  \partial_z f dz, \ \text{where} \ h_i^\star=-h_i^2+\frac{1}{\psi_i H_i^2}
\end{equation} 
and therefore the coordinate surfaces of $f^\star$ are dual surfaces of the coordinate surfaces of $f$. In particular, the existence of a dual system $f^\star$ just follows from the existence of the associated system $\hat{f}$.

In summary, we have proven
\begin{prop}
The coordinate surfaces of a G-system are G-surfaces and the Combescure transforms $\hat{f}$ and $f^\star$ satisfying (\ref{cond_G_system}) consist of the associated and dual surfaces of the coordinate surfaces of $f$, respectively.
\end{prop}
%
%
%
%
%
%
\section{Guichard nets}\label{section_guichard_nets}
\noindent Recall from Section \ref{section_tos} that a Guichard net $(f,H_1,H_2,H_3)$ is a 0-system; hence, by setting $(\varepsilon_1, \varepsilon_2, \varepsilon_3)=(1,1,-1)$, these systems satisfy the \emph{Guichard condition}
\begin{equation}\label{Guichard_cond_eps}
\varepsilon_1 H_1^2 + \varepsilon_2 H_2^2 + \varepsilon_3 H_3^2=0
\end{equation}
and, by differentiating, the following relation
\begin{equation}\label{diff_guichard_cond}
\varepsilon_i \partial_i H_i + \varepsilon_j H_j \beta_{ij}  + \varepsilon_k H_k\beta_{ik} =0
\end{equation}
between the Lam\'e coefficients and the rotational coefficients.
\begin{thm}\label{Guichard_G_surfaces}
A Guichard net is a G-system and, in particular, the coordinate surfaces of a Guichard net are G-surfaces.
\end{thm}
\begin{proof}
Let $(f, H_1, H_2, H_3)$ be a Guichard net. From Lam\'{e}'s first system (\ref{lame_first}) and the Guichard condition we learn that the functions
\begin{equation*}
\pm\sqrt{H_1^2+H_2^2}, \ \ \pm\sqrt{H_3^2-H_1^2} \ \ \ \text{and } \ \ \pm\sqrt{H_3^2-H_2^2}
\end{equation*}
are solutions of the point equation of the corresponding coordinate surface family. Thus, by Proposition \ref{G_point_invariant}, the coordinate surfaces of a Guichard net are G-surfaces and the functions
\begin{equation*}
\varphi_3:=  \frac{H_3}{H_1H_2}, \ \ \varphi_2:= -\frac{H_2}{H_1H_3}, \ \ \varphi_1:=- \frac{H_1}{H_2H_3}
\end{equation*}
induce the corresponding Combescure transformations that transform any coordinate surface to the associated surfaces.

Moreover, Proposition \ref{tos_comb_technical} guarantees that, if suitably chosen, the associated surfaces constitute again a triply orthogonal system: by the Guichard condition and (\ref{diff_guichard_cond}), we obtain
\begin{equation*}
\begin{aligned}
\varphi_1+\varphi_2+\varphi_3=0, \hspace*{2cm} & 
\\[4pt]\partial_j\varphi_j - \varphi_i \partial_j \ln H_k - \varphi_k \partial_j \ln H_i=&
\\\frac{1}{H_iH_jH_k} \{\varepsilon_jH_j \partial_j H_j + \varepsilon_kH_k\partial_j H_k + \varepsilon_i H_i \partial_j H_i\}=& \ 0
\end{aligned}
\end{equation*}
and therefore the functions $\varphi_1$, $\varphi_2$ and $\varphi_3$ induce a 1-parameter family of Combescure transformed systems. By construction, these are the associated systems of the Guichard net $f$.

Thus, equation (\ref{star_from_assoc}) describes the corresponding dual systems of $f$ and a Guichard net is indeed a G-system.
\end{proof}
To gain more insights into the geometry of Guichard nets, we explicitly construct its associated systems:

\begin{prop}\label{prop_associated_tos}
Let $f$ be a Guichard net, then the associated systems $\hat{f}$ of $f$ are 1-systems given by  
\begin{equation}\label{equ_hat}
d\hat{f}=h_1 \partial_x f dx + h_2 \partial_y f dy + h_3 \partial_z f dz,
\end{equation}
where the functions $h_1, h_2$ and $h_3$ are determined by 
\begin{equation}
\begin{aligned}\label{system_associated}
\partial_x h_3&=-\frac{H_2}{H_3}\kappa_{13},
\\\partial_y h_3&=\phantom{-}\frac{H_1}{H_3}\kappa_{23},
\\\partial_z h_3&= -\frac{H_1H_2}{H_3^2} (\kappa_{31}-\kappa_{32}),
\\h_1&=h_3+\frac{H_2}{H_1H_3},
\\[6pt] h_2&=h_3-\frac{H_1}{H_2H_3}.
\end{aligned}
\end{equation}
\end{prop}
\begin{proof} 
Suppose that $(f,H_1,H_2,H_3)$ is a Guichard net. From the proof of Theorem~\ref{Guichard_G_surfaces}, we know that $f$ is a G-system with associated systems induced by the functions
\begin{equation*}
\varphi_3:=  \frac{H_3}{H_1H_2}, \ \ \varphi_2:= -\frac{H_2}{H_1H_3}, \ \ \varphi_1:=- \frac{H_1}{H_2H_3}.
\end{equation*}
Hence, from the system (\ref{system_comb_technical}) in the proof of Proposition \ref{tos_comb_technical} and the Guichard conditions (\ref{Guichard_cond_eps}) and  (\ref{diff_guichard_cond}), we obtain the asserted system (\ref{system_associated}).
\\\\Moreover, these constructed associated systems are 1-systems, since
\begin{align*}
\hat{H}_1^2+\hat{H}_2^2-\hat{H}_3^2= (h_1^2-h_3^2)H_1^2+(h_2^2-h_3^2)H_2^2=1.
\end{align*}
\end{proof}
\noindent Conversely, from the existence of a particular Combescure transformed triply orthogonal system, we can conclude that the Guichard condition is satisfied. 

Thus, we obtain a characterization of Guichard nets:
\begin{thm}\label{thm_char_guichard}
A triply orthogonal system $(f, H_1, H_2, H_3)$ is a Guichard net if and only if there exists a Combescure transformed system $(\hat{f}, \hat{H}_1, \hat{H}_2, \hat{H}_3)$ such that the relations 
\begin{equation}
\begin{aligned}\label{char_guichard}
H_i \hat{H}_j - H_j \hat{H}_i = \varepsilon_k H_k
\end{aligned}
\end{equation}
between the Lam\'{e} coefficients of $f$ and $\hat{f}$ hold. In this case, the system $\hat{f}$ is an associated system of the Guichard net. 
\end{thm}
\begin{proof}
Suppose that $\hat{f}$ is a Combescure transform of $f$ fulfilling (\ref{char_guichard}) and define $h_i:=\frac{\hat{H}_i}{H_i}$. Then we obtain
\begin{equation*}
\frac{1}{H_1H_2H_3} \{ H_1^2+ H_2^2-H_3^2 \}=h_3-h_2+h_1-h_3-h_1+h_2=0
\end{equation*}
and $f$ is indeed a Guichard net.

Conversely, let $f$ be a Guichard net. Then, by Proposition \ref{prop_associated_tos}, there exist associated systems given by (\ref{system_associated}) which fulfill
\begin{equation*}
h_j-h_i=\frac{\varepsilon_k H_k}{H_i H_j}.
\end{equation*}
Hence, any associated system yields a sought-after Combescure transform and the claim is proven.
\end{proof}
\noindent From Proposition \ref{prop_associated_tos}, we obtain, similar to the case of G-surfaces, the following relation between the associated systems: if $\hat{f}_0$ is an associated system induced by the functions $h_1, h_2$ and $h_3$ satisfying system (\ref{system_associated}), then the 1-parameter family $\hat{f}_c$ of associated systems are induced by the functions $h_{ic}:=h_i+c$, where $c \in \mathbb{R}$. Hence, we obtain
\begin{equation*}
\hat{f}_c=\hat{f}_0+ c f.
\end{equation*} 
As a consequence, we conclude that the dual systems of a Guichard net are given by
\begin{equation}\label{dual_guichard_form}
df_c^\star = h_{1c}^{\star} \partial_x f dx + h_{2c}^\star \partial_y f dy + h_{3c}^\star \partial_z f dz, \ \text{where } h_{ic}^\star=-(h_i+c)^2+\frac{\varepsilon_i}{H_i^2}.
\end{equation}
The characterization of Guichard nets given in Theorem~ \ref{thm_char_guichard} then shows, that the dual systems provide a 1-parameter family of (generically) new Guichard nets: 
\begin{prop}
The dual systems of a Guichard net are again Guichard nets.
\end{prop}
\begin{proof}
Let $f$ be a Guichard net with an associated system $\hat{f}$ and the corresponding
 dual system $f^\star$.  We will show that 
\begin{equation*}
H_i^\star \hat{H}_j - H_j^\star \hat{H}_i = - \varepsilon_k H_k^\star,
\end{equation*} 
hence, a Combescure transformation defined by the functions $-h_1, -h_2$ and $-h_3$ satisfies the conditions (\ref{char_guichard}) of Theorem~\ref{thm_char_guichard} and therefore the dual system $f^\star$ is a Guichard net.

To prove this relation, we use $\varepsilon_i\varepsilon_j=-\varepsilon_k$, formula (\ref{dual_guichard_form}) and the relations 
\begin{equation*}
h_i=h_k +  \frac{\varepsilon_j H_j}{H_i H_k} \ \ \text{and} \ \ h_j = h_k - \frac{\varepsilon_i H_i}{H_j H_k}. 
\end{equation*} 
Then we obtain
\begin{equation*}
\begin{aligned}
H_i^\star \hat{H}_j - H_j^\star \hat{H}_i &= h_ih_j \{ -H_i\hat{H}_j - H_j \hat{H}_i \} + \frac{1}{H_iH_j} \{ \varepsilon_i h_jH_j^2 - \varepsilon_j h_i H_i^2 \}
\\&= \varepsilon_k h_k^2H_k - \frac{1}{H_k}= - \varepsilon_k H_k^\star,
\end{aligned}
\end{equation*}
which completes the proof.
\end{proof}
%
%
%
\subsection{Guichard nets with cyclic associated systems} A triply orthogonal system is \emph{cyclic} if two of the coordinate surface families consist of channel surfaces.

In this subsection we employ the geometric consequences for a Guichard net if all its associated systems are cyclic. It will turn out that these Guichard nets are rather special and contain a family of parallel or totally umbilic coordinate surfaces.

To begin, we recall some properties of such systems (cf.\,\cite{eisenhart_treatise, salkowski}) and remark some immediate relations to their associated systems. 
\\\\If the coordinate surfaces of a Lam\'e family are parallel, then the other two families are developable surfaces formed by the normal lines along the curvature lines of the parallel surface family. Thus, two of the rotational coefficients vanish. Conversely, if for example $\beta_{ji}=\beta_{ki}=0$, then we conclude from (\ref{direction_cos}) and (\ref{recover_f}) that the coordinate surfaces $x_i=const$ are parallel. 

Consequently, a triply orthogonal system contains a family of parallel coordinate surfaces if and only if the corresponding coordinate surfaces of a (hence all) Combescure transforms are also parallel. 
\\\\Moreover, taking into account construction (\ref{system_associated}) for the associated systems, we deduce two characterizations for the Guichard nets under consideration.

In the following paragraphs, the functions $h_1$, $h_2$ and $h_3$ fulfill the system (\ref{system_associated}) and therefore induce an associated system of the Guichard net.
\begin{cor}\label{cor_parallel}
Let $f$ be a Guichard net with associated system $\hat{f}$. Then the following are equivalent:
\begin{itemize}
\item[\emph{(i)}] The coordinate surfaces $x_i=const$ of $f$ are parallel.
\item[\emph{(ii)}]  The coordinate surfaces $x_i=const$ of $\hat{f}$ are parallel.
\item[\emph{(iii)}]  $\partial_j h_i = \partial_k h_i =0$.
\end{itemize} 
\end{cor}
\begin{cor}\label{cor_umbilic}
The coordinate surfaces $x_i=const$ of a Guichard net are totally umbilic if and only if $\partial_i h_i = 0$.
\end{cor}
\noindent These two characterizations in terms of the functions $h_1, h_2$ and $h_3$ enable us to understand Guichard nets with cyclic associated systems:
\begin{prop}\label{assoc_cyclic}
A Guichard net contains a family of parallel or totally umbilic coordinate surfaces if and only if all its associated systems are cyclic. 
\end{prop}
\begin{proof}
Recall that a triply orthogonal system is $x_i$-cyclic, where the $x_i$-trajectories provide the circular direction, if and only if 
\begin{equation*}
0=\partial_i \kappa_{ji} = \partial_i \kappa_{ki} \ \ \Leftrightarrow \ \ 0=H_i\partial_i \beta_{ji} - \beta_{ji} \partial_iH_i =  H_i\partial_i \beta_{ki} - \beta_{ki} \partial_iH_i.
\end{equation*}
Thus, suppose that the associated systems are $x_i$-cyclic and denote the Lam\'{e} coefficients of them by $\hat{H}_i^c=(h_i+c)H_i$, where $c \in \mathbb{R}$. Then, we obtain 
\begin{align*}
\text{for } c=0: \ & h_iH_i \partial_i \beta_{ji}-\beta_{ji} \partial_i (h_iH_i)=h_iH_i \partial_i \beta_{ki}-\beta_{ki} \partial_i (h_iH_i)=0 \ \ \text{and}
\\[6pt]\text{for } c \neq 0: \ & h_iH_i \partial_i \beta_{ji}-\beta_{ji} \partial_i (h_iH_i) + cH_i \partial_i \beta_{ji} - c\beta_{ji}\partial_i H_i=
\\ & h_iH_i \partial_i \beta_{ki}-\beta_{ki} \partial_i (h_iH_i) + cH_i \partial_i \beta_{ki} - c\beta_{ki}\partial_i H_i=0,
\end{align*}
which shows that the corresponding Guichard net is $x_i$-cyclic. Therefore, we deduce from the circularity of the cyclic associated systems that
\begin{align*}
0&=\partial_i \kappa_{ji}^{c\star}= \partial_i \frac{\kappa_{ji}}{h_i +c},
\\0&=\partial_i \kappa_{ki}^{c\star}= \partial_i \frac{\kappa_{ki}}{h_i +c}.
\end{align*}
As a consequence of the Corollaries \ref{cor_parallel} and \ref{cor_umbilic}, the coordinate surfaces $x_i=const$ are parallel (if $\kappa_{ji}=\kappa_{ki}=0$) or totally umbilic (if $\partial_i h_i =0$).  

Conversely, assume that the Guichard net has a family of parallel coordinate surfaces, then, by Corollary \ref{cor_parallel}, the corresponding coordinate surfaces of the associated systems are also parallel and therefore the associated systems are cyclic. 

In the other case, if the coordinate surfaces $x_i=const$ of the Guichard net are totally umbilic, then the Guichard net is a cyclic system (cf.\,\cite[\S 4]{uhj_guichard}) and, by Corollary \ref{cor_umbilic}, we obtain $h_i=h_i(x_j, x_k)$. Thus, the principal curvatures $\kappa_{ji}^\star$ and $\kappa_{ki}^\star$ do not depend on $x_i$ and therefore the associated systems are $x_i$-cyclic.   
\end{proof}
\noindent Therefore, the associated systems of a Guichard net are triply orthogonal systems consisting of Dupin cyclides if and only if all coordinate surface families of the Guichard net are totally umbilic or parallel. 

This leads to an interesting example which shows that, in general, the associated systems do not admit a reparametrization into a Guichard net. 
\begin{ex}\label{example_guichard}
The so-called 6-sphere coordinates, obtained as inversion of the Cartesian coordinates and consisting of spheres that are all tangent to a fixed point, form a Guichard net. It is determined by the induced metric
\begin{equation*}
I=\frac{1}{(x^2+y^2+2z^2)^2} \big\{ dx^2 + dy^2 + 2 dz^2 \big\}.
\end{equation*} 
By solving system (\ref{system_associated}), we conclude that the induced metrics of the associated systems become
\begin{equation*}
\hat{I}_c= \Big \{\frac{c + \sqrt{2}(y^2+z^2)}{x^2+y^2+2z^2}\Big\}^2 dx^2 + 
\Big \{\frac{c - \sqrt{2}(x^2+z^2)}{x^2+y^2+2z^2}\Big\}^2 dy^2 + 
\Big \{\frac{\sqrt{2}c - x^2 +y^2}{x^2+y^2+2z^2}\Big\}^2 dz^2,
\end{equation*}
where $c \in \mathbb{R}$. The principal curvatures of these systems, as well as Proposition \ref{assoc_cyclic}, then show that all coordinate surfaces are Dupin cyclides. 
\\\\Moreover, any Dupin cyclide is an isothermic surface, that is, it admits conformal curvature line coordinates. Hence, the associated systems of the 6-sphere coordinates consist of isothermic surfaces and  are therefore isothermic triply orthogonal systems in the sense of Darboux \cite{darboux_1878}. 

Taking into account the explicit classification of isothermic Guichard nets given in \cite[Chap.\,III]{iso_guichard, my_thesis}, it follows that the associated systems of the 6-sphere coordinates do not admit a reparametrization into a Guichard net.
\\\\By (\ref{dual_guichard_form}), the Lam\'e coefficients of the dual systems of the 6-sphere coordinates read
\begin{align*}
H_{1c}^\star&=\frac{(x^2+y^2+2z^2)^2-(c+\sqrt{2}(y^2+z^2))^2}{x^2+y^2+2z^2},
\\[4pt]H_{2c}^\star&=\frac{(x^2+y^2+2z^2)^2-(c-\sqrt{2}(x^2+z^2))^2}{x^2+y^2+2z^2},
\\[4pt]H_{3c}^\star&=\frac{-\sqrt{2} \left(x^2+y^2+2 z^2\right)^2-\sqrt{2}\left(c+\frac{1}{\sqrt{2}}(y^2-x^2)\right)^2}{x^2+y^2+2 z^2}
\end{align*} 
\ \\[4pt] and we obtain a 1-parameter family of new non-cyclic Guichard nets. Since the seed Guichard net $f$ is totally cyclic, the torsion
\begin{equation*}
\tau_i^\star= \frac{\kappa_{ji}^\star\partial_i\kappa_{ki}^\star-\kappa_{ki}^\star\partial_i\kappa_{ji}^\star}{(\kappa_{ji}^\star)^2+(\kappa_{ki}^\star)^2}
\end{equation*}
of any $x_i$-coordinate curve in the dual systems vanishes and therefore the dual Guichard nets consist of coordinate surfaces with planar curvature lines. 

However, the dual systems do not contain parallel or totally umbilic surfaces, which is also reflected in the fact that not all associated systems of the dual systems are cyclic (cf.\,Proposition \ref{assoc_cyclic}). Thus, the geometry of the associated systems of the original Guichard net and the associated systems of its dual systems change considerably.
\end{ex}
\begin{center}
\hspace*{-1cm}\includegraphics[scale=0.25]{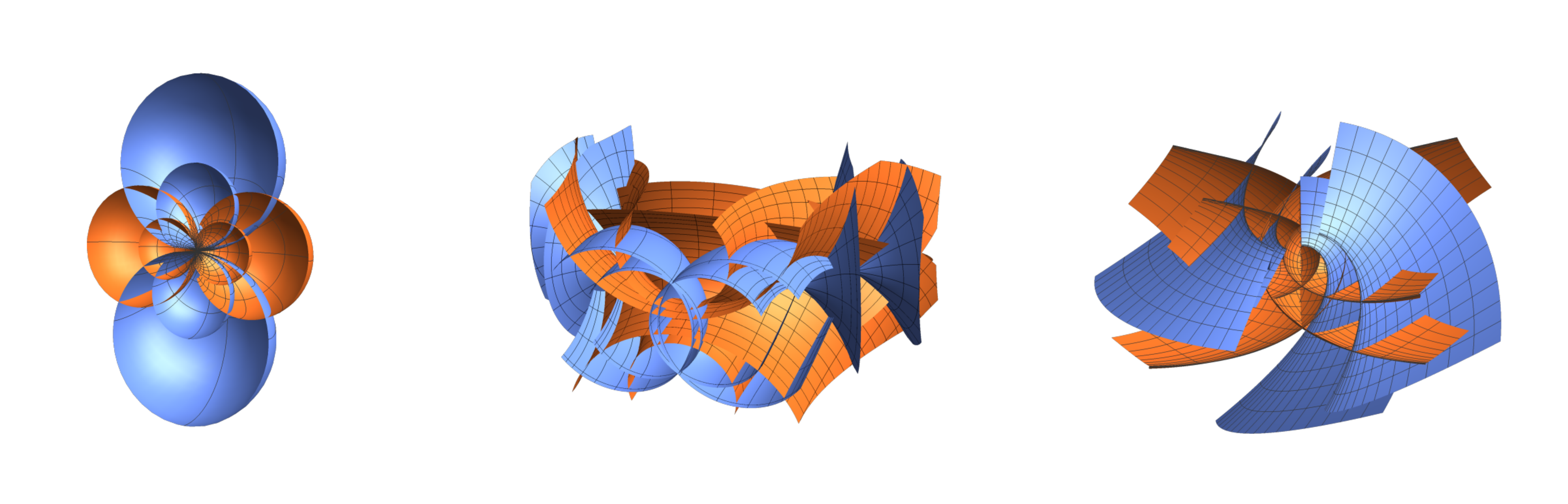}
\\[6pt]\small\emph{Figure 3.} The 6-sphere coordinates parametrized as Guichard net (\emph{left}) with an associated system (\emph{middle}) and the corresponding dual system (\emph{right}).
\end{center}
%
%
%
%
%
%
\section{B\"acklund-type transformations of Guichard nets}\label{section_trafo}
\noindent Classically \cite[\S 4]{bianchi_1918}, two triply orthogonal systems form a \emph{Ribaucour pair} if any two corresponding coordinate surfaces are Ribaucour transformations of each other; that is, two corresponding surfaces envelop a common sphere congruence so that curvature lines correspond. 

In this paper we are interested in Ribaucour transformations that preserve the Guichard condition of a Guichard net (cf.\,\cite{bianchi_1918, bianchi_1919, duality}). 
\\\\To begin with, we summarize some facts about the Ribaucour transformation of triply orthogonal systems in general, which will then be used for the particular case of Guichard nets and their associated systems. 

More details and results in this realm can be found in the two classical monographs \cite{bianchi_1918} and \cite{bianchi_1919}.

\subsection{Ribaucour transforms of triply orthogonal systems} 
Any Ribaucour transform $(f', H'_1, H'_2, H'_3)$ of a triply orthogonal system $(f, H_1, H_2, H_3)$ is uniquely described by four functions $\gamma_1, \gamma_2, \gamma_3$ and $\varphi$ satisfying 
\begin{equation}\label{rib_comp}
\partial_i \gamma_j = \beta_{ji} \gamma_i \ \ \text{and} \ \ \partial_i \varphi = H_i \gamma_i, 
\end{equation}
where the data of the Ribaucour transform reads  
\begin{align*}
f'&=f-\frac{2 \varphi}{A}(\gamma_1 N_1 + \gamma_2 N_2+\gamma_3 N_3),
\\H'_i&= H_i - \frac{2\varphi\theta_i}{A},
\\\beta'_{ij}&=\beta_{ij}-\frac{2 \gamma_i \theta_j}{A}
\end{align*}
and
\begin{align*}
A&:=\gamma_1^2+\gamma_2^2+\gamma_3^2,
\\\theta_i&:=\partial_i \gamma_i+\beta_{ji}\gamma_j+\beta_{ki}\gamma_k.
\end{align*}
Here, $N_i$ denotes the unit normals of the coordinate surfaces $x_i=const$ of $f$.
\\\\Note that, by Lam\'e's equations, for three given functions $\gamma_1, \gamma_2$ and $\gamma_3$ satisfying (\ref{rib_comp}) there exists a 1-parameter family of suitable functions $\varphi$, which define a 1-parameter family of Ribaucour transforms.
Since the radii of the enveloped Ribaucour sphere congruences for the coordinate surface family $x_i=const$ are given by
\begin{equation*}
R_i=-\frac{\varphi}{\gamma_i},
\end{equation*}
the 1-parameter family of functions $\varphi$ indeed determines different Ribaucour transformations. 
\\\\
As observed by Bianchi in \cite[\S 8]{bianchi_1918}, any Ribaucour transformation of a  triply orthogonal system can be decomposed into two Combescure transformations and an inversion in the unit sphere. Since this construction will turn out to be a crucial tool in the transformation of Guichard nets, we recall it here in detail.

Suppose $f'$ is a Ribaucour transformation of a triply orthogonal system $f$ described by the functions $\gamma_1$, $\gamma_2$, $\gamma_3$ and $\varphi$.
If we denote the unit normals of the coordinate surfaces $x_i=const$ of $f$ by $N_i$, then  the function
\begin{equation*}
\bar{f}:=\gamma_1 N_1 + \gamma_2 N_2 + \gamma_3 N_3
\end{equation*}
defines a Combescure transform $\bar{f}=:\mathcal{C}_1(f)$ of $f$ with induced metric
\begin{equation*}
\bar{I}=\theta_1^2 dx^2 + \theta_2^2 dy^2 + \theta_3^2 dz^2.
\end{equation*}
Moreover, the function $\bar{\varphi}:=\frac{1}{2}(A-1)$ fulfills condition (\ref{rib_comp}) with respect to the induced metric $\bar{I}$ and therefore the functions $(\gamma_1, \gamma_2, \gamma_3, \bar{\varphi})$ determine a Ribaucour transformation $\bar{f}'$ of $\bar{f}$:
\begin{equation*}
\bar{f}'=\bar{f}-\frac{2\bar{\varphi}}{A}\bar{f}=\frac{1}{A}\bar{f}=\frac{1}{|\bar{f}|^2}\bar{f}.
\end{equation*}
Thus, this Ribaucour transformation is an inversion $\iota$ in the unit sphere and the induced metric of $\bar{f}'=(\iota\circ\mathcal{C}_1)(f)$ is given by
\begin{equation*}
\bar{I}'=\frac{1}{A^2}\bar{I}.
\end{equation*}
Since the metric coefficients of $\bar{I}'$ fulfill
\begin{equation}\label{rib_trafo_inv}
\partial_j \frac{\theta_i}{A}=\{ \beta_{ji}-\frac{2\gamma_j\theta_i}{A} \}\frac{\theta_j}{A}=\mathcal{R}(\beta_{ji})\frac{\theta_j}{A},
\end{equation}
the transformed system $(\iota\circ\mathcal{C}_1)(f)$ is a Combescure transformation of $f'$ and we indeed gain the sought-after decomposition of the Ribaucour transformation into two Combescure transformations and an inversion.

The constructed systems $(\bar{f},\bar{f}')$ will be called the \emph{induced} Combescure transformations of the Ribauocur pair $(f,f')$. 

\bigskip

\noindent Conversely, a Combescure transformed system $\bar{f}$ of $f$ \emph{induces} a 1-parameter family of Ribaucour transforms of $f$: suppose that $\bar{f}$ is an arbitrary Combescure transform of $f$ and let $N_i$ denote the common unit normals of its coordinate surfaces $x_i=const$. Then the functions 
\begin{equation*}
\gamma_i:=\bar{f}\cdot N_i  
\end{equation*}
and a solution $\varphi$ of the integrable system
\begin{equation}\label{ref_phi}
\partial_i \varphi = H_i \gamma_i,
\end{equation}
where $H_i$ are the metric coefficients of $f$, define a Ribaucour transform $f'$ of $f$.  

Since the function $\varphi$ is, by the equations (\ref{ref_phi}), uniquely defined up to an additive constant, we obtain a 1-parameter family of induced Ribaucour transforms. 
\\\\Furthermore, since the Ribaucour transformed rotational coefficients $\beta_{ij}'$ only depend on the functions $\gamma_1, \gamma_2$ and $\gamma_3$, two Ribaucour transforms that are induced by the same system $\bar{f}$ share the same rotational coefficients. Hence, we obtain the following lemma (cf.\,Definition \ref{def_combescure}):
\begin{lem}\label{lemma_comb_ind}
If two Ribaucour transforms $f'_1$ and $f'_2$ are induced by the same triply orthogonal system, then $f'_1$ and $f'_2$ are related by a Combescure transformation.  
\end{lem}
%
%
\subsection{B\"{a}cklund-type transformations of Guichard nets}
The purpose of this subsection is to geometrically understand the interplay between Ribaucour pairs of Guichard nets and the induced Combescure transforms in the decomposition discussed in the previous paragraphs.
\begin{prop}\label{induced_comb_guichard}
Let $(f, f')$ be a Ribaucour pair of Guichard nets, then the induced Combescure transforms $\bar{f}$ and $\bar{f}'$ are $\alpha^2 |\bar{f}|^2$- and $\alpha^2 |\bar{f}'|^2$-systems, $\alpha \in \mathbb{R}$, respectively. 
\end{prop}
\begin{proof}
By \cite[\S 11]{bianchi_1919}, we know that a Ribaucour pair $(f,f')$ of Guichard nets is generated by the data $(\gamma_1,\gamma_2,\gamma_3,\varphi)$ fulfilling
\begin{align}
\partial_j \gamma_i&= \beta_{ij}\gamma_j, \ \ \ \theta_i=\alpha \bar{\gamma}_i
\\\partial_j \bar{\gamma}_i&= \beta_{ji}\bar{\gamma}_j, \ \ \ \bar{\theta}_i=\alpha \gamma_i,
\\A:=\gamma_1^2+\gamma_2^2+\gamma_3^2&=\bar{\gamma}_1^2+\bar{\gamma}_2^2-\bar{\gamma}_3^2=:\bar{A},
\end{align} 
where $\bar{\theta}_i:=\varepsilon_i\partial_i\bar{\gamma}_i+\varepsilon_j \beta_{ij} \bar{\gamma}_j +\varepsilon_k \beta_{ik} \bar{\gamma}_k$ and
\begin{equation*}
\varphi:=\frac{1}{\alpha}(\varepsilon_i H_i \bar{\gamma}_i+\varepsilon_j H_j \bar{\gamma}_j+\varepsilon_k H_k \bar{\gamma}_k).
\end{equation*}
Therefore, the Lam\'e coefficients $(\theta_1, \theta_2, \theta_3)$ of the induced Combescure transform 
\begin{equation*}
\bar{f}=\gamma_1N_1+\gamma_2 N_2+\gamma_3N_3
\end{equation*}
satisfy the condition 
\begin{equation*}
\theta_1^2+\theta_2^2-\theta_3^2=\alpha^2\bar{A}=\alpha^2 A=\alpha^2|\bar{f}|^2.
\end{equation*}
Moreover, taking into account the change of the induced metric after an inversion in the unit sphere, completes the proof.
\end{proof} 
%
%
Conversely, we can utilize this fact to construct Ribaucour pairs of Guichard nets and obtain in this way a B\"{a}cklund-type transformation for Guichard nets:
\begin{thm}\label{thm_trafo_guichard}
Let $(f, H_1, H_2, H_3)$ be a Guichard net and $(\bar{f}, \bar{H}_1,\bar{H}_2,\bar{H}_3)$ a Combescure transformed $\alpha^2 |\bar{f}|^2$-system of $f$. Then the family of Ribaucour transformations of $f$ induced by $\bar{f}$ contains exactly one Guichard net $\mathcal{R}(f)$, which is given by the data
\begin{equation}
\begin{aligned}\label{data_rib_guichard}
\mathcal{R}(f)&=f- \frac{2\varphi}{|\bar{f}|^2}\bar{f},
\\\mathcal{R}(H_i)&= H_i - \frac{2\varphi}{|\bar{f}|^2}\bar{H}_i,
\end{aligned}
\end{equation}
where $\varphi:=\frac{1}{\alpha^2}\{ H_1 \bar{H}_1+ H_2 \bar{H}_2- H_3 \bar{H}_3\}$.

This Ribaucour transformation is said to be a \emph{B\"{a}cklund-type transformation of $f$ with respect to the system $\bar{f}$}.
\end{thm}
\begin{proof}
Let $\bar{f}$ be a Combescure related $\alpha^2 |\bar{f}|^2$-system of the Guichard net $f$. Then, by using the differentiated Guichard condition, we conclude that the 1-parameter family of induced Ribaucour transformations of $f$ are defined by the functions $\gamma_i:=\bar{f} \cdot N_i$ and 
\begin{equation*}
\varphi_\lambda:=\frac{1}{\alpha^2}\{ H_1\bar{H}_1 + H_2\bar{H}_2 - H_3\bar{H}_3 \}+\lambda, \ \ \lambda \in \mathbb{R},
\end{equation*}
which satisfy the conditions (\ref{rib_comp}). Therefore, these Ribaucour transforms are described by 
\begin{align}\label{rib_guichard}
\mathcal{R}(f)&= f - \frac{2\varphi_\lambda}{A}(\gamma_1N_1 + \gamma_2 N_2 + \gamma_3 N_3)=f-\frac{2\varphi_\lambda}{|\bar{f}|^2}\bar{f}
\\\mathcal{R}(H_i)&= H_i - \frac{2 \varphi_\lambda \theta_i}{A} = H_i - \frac{2 \varphi_\lambda}{|\bar{f}|^2}\bar{H}_i.
\end{align}
Moreover, the Lam\'e coefficients satisfy
\begin{align*}
\mathcal{R}(H_1)^2+&\mathcal{R}(H_2)^2-\mathcal{R}(H_3)^2= 
\\[4pt]&\frac{4 \varphi^2}{|\bar{f}|^2}(\bar{H}_1^2+\bar{H}_2^2-\bar{H}_3^2) - \frac{4 \varphi}{|\bar{f}|^2} (H_1\bar{H}_1 + H_2\bar{H}_2 - H_3\bar{H}_3)=\frac{4\alpha^2\lambda}{|\bar{f}|^2}\varphi.
\end{align*}
Thus, an induced Ribaucour transform determined by (\ref{rib_guichard}) is again a Guichard net if and only if $\lambda=0$. 
\end{proof}
By Proposition \ref{induced_comb_guichard} and Theorem~ \ref{thm_trafo_guichard}, we have reduced the B\"{a}cklund-type transformation for Guichard nets to the existence of Combescure transformed $\alpha^2|\bar{f}|^2$-systems of the Guichard nets. Thus, after an $\alpha^2|\bar{f}|^2$-system is determined, the B\"acklund-type transformation of any Combescure related Guichard net is obtained by simple computations from ($\ref{data_rib_guichard}$).
\\\\Analogous proofs as given for Proposition \ref{induced_comb_guichard} and Theorem~\ref{thm_trafo_guichard} lead to a B\"acklund-type transformation for the class of $\lambda$-systems, where $\lambda \in \mathbb{R}$ is a constant:
\begin{cor}\label{cor_const_backlund}
Let $(f, H_1, H_2, H_3)$ be a $\lambda$-system, $\lambda \in \mathbb{R}$, then any Ribaucour transformed $\lambda$-system $\mathcal{R}(f)$ of $f$ is induced by a Combescure related $\alpha^2|\bar{f}|^2$-system $(\bar{f},\bar{H}_1,\bar{H}_2,\bar{H}_3)$ and determined by
\begin{equation}
\begin{aligned}\label{data_rib_const}
\mathcal{R}(f)&=f- \frac{2\varphi}{|\bar{f}|^2}\bar{f},
\\\mathcal{R}(H_i)&= H_i - \frac{2\varphi}{|\bar{f}|^2}\bar{H}_i,
\end{aligned}
\end{equation}
where $\varphi:=\frac{1}{\alpha^2}\{ H_1 \bar{H}_1+ H_2 \bar{H}_2- H_3 \bar{H}_3\}$.
\end{cor}
Based on these observations, we prove a permutability theorem for the associated and dual systems of a Guichard net and the B\"acklund-type transformation:
\begin{thm}
Let $f$ be a Guichard net with associated systems $\hat{f}_c$ and dual systems $f^\star_c$. Then the B\"acklund-type transforms $\mathcal{R}(\hat{f}_c)$ and $\mathcal{R}(f_c^\star)$ induced by an $\alpha^2|\bar{f}|^2$-system provide associated and dual systems of the Ribaucour transformed Guichard net $\mathcal{R}(f)$ induced by the same $\alpha^2|\bar{f}|^2$-system:
\begin{equation*}
\mathcal{R}(\hat{f}_c)=\widehat{\mathcal{R}(f)}_c  \ \ \text{ and } \ \ \mathcal{R}(f_c^\star)= \mathcal{R}(f)_c^\star.
\end{equation*}
\end{thm}
\ \\ \ 
\begin{center}\begin{tikzcd}[column sep=4em]
\bm{\bar{f}} \arrow[rd,no head, dashed,"\mathcal{C}"] \arrow[rdd,no head, dashed,"\mathcal{C}"]\arrow[rddd,no head, dashed,"\mathcal{C}"] \arrow[rrr,dash,"\text{inversion}"] & & & \bm{\bar{f}'} \arrow[ld,no head, dashed,swap,"\mathcal{C}"] \arrow[ldd,no head, dashed,swap,"\mathcal{C}"]\arrow[lddd,no head, dashed,swap,"\mathcal{C}"] \\[-1em]
   & \bm{f} \arrow[d,dash,"\mathcal{C}"] \arrow[r,dash,"\mathcal{R}"] & \bm{\mathcal{R}(f)}\arrow[d,dash,swap,"\mathcal{C}"] &   \\
      & \bm{\hat{f}} \arrow[d,dash,"\mathcal{C}"]\arrow[r,dash,"\mathcal{R}"] & \bm{\widehat{\mathcal{R}(f)}}\arrow[d,dash,swap,"\mathcal{C}"] &   \\
         & \bm{f^{\star}} \arrow[r,dash,"\mathcal{R}"] & \bm{\mathcal{R}(f)^{\star}} &   \\
\end{tikzcd}
\\\small\emph{Figure 4.} Permutability theorem for a Guichard net $f$ with associated system $\hat{f}$ and dual system $f^\star$, where all Ribaucour transformations $\mathcal{R}$ are induced by the $\alpha^2|\bar{f}|^2$-system $\bar{f}$ and $\alpha^2|\bar{f}'|^2$-system $\bar{f}'$, respectively.
\end{center}
\begin{proof}
Firstly observe that, due to Lemma \ref{lemma_comb_ind}, the systems $\mathcal{R}(f)$, $\mathcal{R}(\hat{f}_c)$ and $\mathcal{R}(f_c^\star)$ are again Combescure transformations of each other, because they are induced by the same $\alpha^2|\bar{f}|^2$-system.

Furthermore, let us denote the Lam\'e coefficients of the Guichard net $f$ and an associated system $\hat{f}$ with the corresponding dual system $f^\star$ by $(H_i)_i$, $(\hat{H}_i)_i$ and $(H^\star_i)_i$, respectively. Thus, by Theorem~\ref{Guichard_G_surfaces}, the Lam\'e coefficients satisfy
\begin{equation*}
H_iH_j^\star+H_jH_i^\star=-2\hat{H}_i\hat{H}_j.
\end{equation*}
and we obtain the following relations 
\begin{equation}\label{phi_relations}
\begin{aligned}
\varphi H_i^\star+\varphi^\star H_i+2\hat{\varphi}\hat{H}_i=\frac{2}{\alpha^2}\bar{H}_j,
\\\varphi \varphi^\star + \hat{\varphi}^2=\frac{1}{\alpha^2}|\bar{f}|^2,
\end{aligned}
\end{equation}
where
\begin{equation*}
\begin{aligned}
\varphi=\frac{1}{\alpha^2} \sum \varepsilon_i H_i\bar{H}_i, \ \ 
\hat{\varphi}= \frac{1}{\alpha^2} \sum \varepsilon_i \hat{H}_i\bar{H}_i \ \ \text{and} \ \ \varphi^\star=\frac{1}{\alpha^2} \sum \varepsilon_i H^\star_i\bar{H}_i.
\end{aligned}
\end{equation*}
These functions, together with the functions $\gamma_i=f \cdot N_i$, $\hat{\gamma}_i=\hat{f}\cdot N_i$ and $\gamma^\star_i=f^\star\cdot N_i$, determine the induced Ribaucour transforms $\mathcal{R}(f)$, ${R}(\hat{f}_c)$ and $\mathcal{R}(f_c^\star)$ (see Theorem~\ref{thm_trafo_guichard} and Corollary~\ref{cor_const_backlund}). 

A straightforward computation using the relations (\ref{phi_relations}) then shows that the Lam\'e coefficients satisfy the conditions
\begin{align*}
\mathcal{R}(H_i)\mathcal{R}(H^\star_j)-\mathcal{R}(H_j)\mathcal{R}(H^\star_i)= - 2\mathcal{R}(\hat{H}_i)\mathcal{R}(\hat{H}_j).
\end{align*}
Thus, by Theorem~\ref{Guichard_G_surfaces}, the systems ${R}(\hat{f}_c)$ and $\mathcal{R}(f_c^\star)$ form an associated and a dual system of ${R}(f)$ and the claim is therefore proven.
\end{proof}
%
%
%
%
%
%
\ \\ \bibliographystyle{abbrv}
\bibliography{mybib}
%
%
%
%
%
\end{document}